\documentclass[11pt,twoside,a4paper]{article}
\usepackage{amsmath,amssymb,mathrsfs}
\usepackage{amsthm}
\usepackage[utf8x]{inputenc}
\usepackage{bm}
\usepackage{avant}
\usepackage[english]{babel}
\usepackage{thmtools,thm-restate}
\usepackage[nottoc]{tocbibind}
\usepackage{hyperref}
\usepackage{graphicx}
\usepackage{enumitem}
\usepackage{tikz}
\usepackage{bbm}
\usepackage[inner=2.4cm,outer=2.4cm,top=2.8cm,bottom=2.8cm]{geometry}
\usepackage{mathtools}
\usepackage{makeidx}
\usepackage{bold-extra}
\usepackage{parskip} 
 \usepackage{todonotes}

\usepackage{xcolor}

\usepackage{graphicx} 

\definecolor{mygreen}{RGB}{31,139,34}

\DeclareMathOperator{\id}{id}

\DeclareMathOperator{\supp}{supp}
\newcommand{\norm}[1]{\left\lVert#1\right\rVert}

\newcommand{\Leb}{\mathscr{L}}
\newcommand{\R}{\mathbb{R}}

\newcommand{\p}{\mathtt p} 
\newcommand{\de}{\ensuremath{\, \mathrm d}} 

\newcommand{\suchthat}{\ensuremath{\,:\,}} 
\newcommand\restr[2]{{
  \left.\kern-\nulldelimiterspace 
  #1 
  \right|_{#2} 
  }}

\newcommand{\CD}{\mathsf{CD}}

\newcommand{\X}{\mathsf{X}} 
\newcommand{\di}{\mathsf d} 
\newcommand{\m}{\mathfrak m} 

\DeclareMathOperator{\Geo}{Geo}

\newcommand{\OptPlans}{\mathsf{OptPlans}}
\newcommand{\Prob}{\mathscr{P}}
\newcommand{\ProbTwo}{\mathscr{P}_2}

\usepackage{titlesec}
\titleformat{\section}
  {\scshape\large\centering}
  {\thesection}{0.5em}{}
\titleformat{\subsection}
  {\scshape\centering}
  {\thesubsection}{0.4em}{}  
  
\author{
Shucheng Li \footnote{Jesus College, University of Oxford. \textit{E-mail}:  \href{mailto:shucheng.li@jesus.ox.ac.uk}{shucheng.li@jesus.ox.ac.uk}}, \ Mattia Magnabosco\footnote{Mathematical Institute, University of Oxford. \textit{E-mail}:  \href{mailto:mattia.magnabosco@maths.ox.ac.uk}{mattia.magnabosco@maths.ox.ac.uk}}, \ Timo Schultz\footnote{Department of Mathematics and Statistics, University of Jyv\"askyl\"a., \textit{E-mail}: \href{mailto:timo.m.schultz@jyu.fi}{timo.m.schultz@jyu.fi} }}
\title{Metric conditions that guarantee existence and uniqueness of Optimal Transport maps}
\date{}

\newtheoremstyle{remark}
        {10pt}
        {10pt}
        {}
        {}
        {\itshape}
        {.}
        {.4em}
        {}

\newtheoremstyle{proof}
        {10pt}
        {10pt}
        {}
        {}
        {\itshape}
        {.}
        {.4em}
        {}
        
\newtheoremstyle{definition}
        {10pt}
        {10pt}
        {}
        {}
        {\bfseries}
        {.}
        {.4em}
        {}

\newtheoremstyle{theorem}
        {10pt}
        {10pt}
        {}
        {}
        {\bfseries}
        {.}
        {.4em}
        {}

\theoremstyle{theorem}
\newtheorem{theorem}{Theorem}[section]
\newtheorem{prop}[theorem]{Proposition}
\newtheorem{corollary}[theorem]{Corollary}
\newtheorem{lemma}[theorem]{Lemma}

\theoremstyle{definition}
\newtheorem{definition}[theorem]{Definition}

\theoremstyle{remark}
\newtheorem{remark}[theorem]{Remark}

\theoremstyle{proof}
\newtheorem*{pro}{Proof}
 {\popQED\end{pro}}

\makeindex
\begin{document}
\maketitle

\begin{abstract}
    We investigate metric conditions that allow to prove existence and uniqueness of a map solving the Monge problem between two marginals in a metric (measure) space, proving two main results. Firstly, we introduce a nonsmooth version of the Riemannian twist condition that we call local metric twist condition, showing, under this assumption on the cost function, existence and uniqueness of optimal transport maps. Secondly, we prove the same result for cost equal to $\di^2$ in a metric space $(\X,\di)$ satisfying a quantitative non-branching assumption, that we call locally-uniformly non-branching. 
\end{abstract}

\section{Introduction}

In recent years, the theory of optimal transport has proved effective in many areas of mathematics, from metric geometry to applied mathematics. The optimal transport problem, following the original approach proposed by Monge, reads as follows: given two topological spaces $X,Y$, two probability measures $\mu\in\Prob(X)$, $\nu\in\Prob(Y)$ and a non-negative Borel cost function $c:X\times Y \to  [0,\infty]$, solve the following minimisation problem:
\begin{equation} \label{MongeFormulation}
	\inf\left\{ \int_X c(x,T(x)) \de \mu(x) \suchthat \text{$T\colon X\to Y$ Borel, $T_\# \mu =\nu$} \right\}. \tag{M}
\end{equation}
In many applications it is interesting not only to find the infimum, but also to identify eventual minimisers, which are called optimal transport maps. However, the infimum in \eqref{MongeFormulation} is not realised in general and thus optimal transport maps do not exist in many cases. Therefore, for the development of a general theory, it is necessary to introduce a more general formulation, due to Kantorovich. Defining the set of admissible transport plans from $\mu$ to $\nu$ as
\begin{equation*}
\mathsf{Adm}(\mu,\nu) := \{ \pi\in\Prob(X\times Y) \suchthat (\p_X)_\#\pi = \mu \,\,\text{and} \,\, (\p_Y)_\#\pi = \nu\},
\end{equation*}
the Kantorovich's formulation of the optimal transport problem is the following:
\begin{equation} \label{KantorovichFormulation}
	\inf \left\{ \int_{X\times Y} c(x,y) \de \pi(x,y) \suchthat \pi\in\mathsf{Adm}(\mu,\nu) \right\}. \tag{K}
\end{equation}
In this version of the optimal transport problem the minimum is attained under fairly general assumptions. In particular, this happens whenever $X$ and $Y$ are Polish spaces and the cost function $c:X\times Y \to [0,\infty]$ is lower semicontinuous. Provided that the infimum is finite, the minimisers of the Kantorovich's formulation are called optimal transport plans and the set containing them all will be denoted by $\OptPlans(\mu,\nu)$. An admissible transport plan $\pi\in\mathsf{Adm}(\mu,\nu) $ is said to be induced by a map if there exists a $\mu$-measurable map $T:X \to Y$ such that $\pi=(\id,T)_\# \mu$. Notice that $\pi=(\id,T)_\# \mu$ is an optimal transport plan if and only if $T$ is an optimal transport map.

The problem of addressing existence and uniqueness of optimal transport maps has been deeply studied. In the smooth setting, many of the positive results are obtained assuming the cost function $c$ to satisfy the so-called \emph{twist condition}. In a Riemannian manifold $(M,g)$ a cost function $c:M \times M \to [0,\infty]$ is said to satisfy the twist condition if the function
\begin{equation}\label{eq:twist}
    y \mapsto \nabla_x c\left(x_0, y\right) \text{ is injective for every }x_0\in M.
\end{equation}
In a series of works it has been proved that, whenever $\mu \ll \text{vol}_g$ and the cost function $c$ satisfies the twist condition and some mild differentiability/concavity assumption, there exists a unique optimal transport map (for \eqref{MongeFormulation}), see \cite{Figalli, FathiFigalli,KimMcC,villani2008optimal}. We remark that the absolutely continuity assumption on the first marginal $\mu$ (or any other weakened version of it) is necessary even in the simplest setting, i.e. $X=Y=\R^n$ and $c(x,y)=\norm{x-y}^2$. 

Other notable results regarding existence and uniqueness of optimal transport maps have been proved in the nonsmooth setting of metric measure spaces, i.e. triples $(\X,\mathsf{d},\mathfrak{m})$, where $(\X,\mathsf{d})$ is a Polish metric space and $\mathfrak{m}$ is a locally finite Borel measure on $\X$. In this framework, it is particularly interesting to study the optimal transport problem with $X=Y=\X$ and $c(x,y)=\di(x,y)^2$, as it is the fundamental tool to formulate the groundbreaking theory of $\CD(K,N)$ spaces developed by Sturm \cite{sturmI, sturmII} and Lott-Villani \cite{lott-villani}. The first results proving existence and uniqueness of optimal transport maps for this problem were obtained by Brenier \cite{brenier} in the Euclidean setting, by McCann \cite{McCann} and \cite{Gigli2} in the Riemannian setting and by Ambrosio-Rigot \cite{AmbrosioRigot} and Figalli-Rifford \cite{FigalliRifford} in the sub-Riemannian setting. In nonsmooth spaces, existence and uniqueness of optimal transport maps is usually proved under a synthetic curvature bound and a non-branching assumption. As examples, we mention the following works: \cite{Bertrand,Bertrand_habi,RajalaSchultz} for Alexandrov spaces, \cite{Gigli, RajalaSturm,GRS,cavallettimondino,Schultz,MagnaboscoRigoni} under various curvature-dimension condition and \cite{CavallettiHuesmann,Kell} under quantitative properties on the reference measure. Moreover we remark that, as shown by Rajala in \cite{Rajala} (see also \cite{Magnabosco1,Magnabosco2}), the non-branching assumption is necessary to prove the uniqueness of the transport map.

In this paper, we investigate metric conditions that are sufficient to guarantee existence and uniqueness of optimal transport maps in metric measure spaces. In Section \ref{sec:metrictwist}, we introduce the \emph{local metric twist condition} (in short \textbf{[LMTC]}) for a cost function on a metric space. This is a consistent nonsmooth generalisation of the classical twist condition \eqref{eq:twist} and it is sufficient to ensure the existence and uniqueness of optimal transport maps.

\begin{theorem}\label{thm:main1}
    Let $(\X,\di, \m)$ be a proper metric measure space with $\m$ doubling on $\X$ and let $c:\X\times \X\to [0,\infty]$ be any lower-semi-continuous cost function that satisfies the local metric twist condition. Then, for every pair of measures $\mu,\nu\in \Prob(\X)$ with $\mu \ll \m$, if $\OptPlans(\mu,\nu) \neq \emptyset $ there exists a unique optimal transport map for the problem \eqref{MongeFormulation}.
\end{theorem}
Recall that a measure $\m$ is called locally doubling if for any $x_0\in X$ there exists $R>0$ and a constant $C>0$ such that
\begin{align*}
    \m(B_{2r}(x))\le C \m(B_r(x))
\end{align*}
for every $B_{2r}(x)\subset B_R(x_0)$.

We also introduce a pointwise (weaker) version of \textbf{[LMTC]}, called \emph{pointwise metric twist condition}, showing that under a regularity assumption (see Definition \ref{def:C1(Met)}) it is sufficient to imply \textbf{[LMTC]}. We highlight that, to our knowledge, Theorem \ref{thm:main1} is the first result of this type for general cost (not necessarily dependent on the underlying distance) in general metric measure spaces.

In Section \ref{sec:main2}, we prove existence and uniqueness of optimal transport maps for the problem with $X=Y=(\X,\di)$ and cost $c(x,y)=\di(x,y)^2$. Our main assumptions are that the first marginal $\mu$ is \emph{densely scattered} and the underlying metric space $(\X,\di)$ is \emph{locally-uniformly non-branching}. The former replaces the usual absolute continuity assumption and asks $\mu$ to give mass to every geodesic cone, see Definitions \ref{def:cone} and \ref{def:denselyscattered}. The latter is a quantitative non-branching condition that resembles (but is weaker than) an Alexandrov-type lower curvature bound. In particular, we prove the following theorem:  

\begin{theorem}\label{thm:main2}
    Let $(\X,\di)$ be a locally-uniformly non-branching metric space and let $\mu\in \ProbTwo(\X)$ be a densely scattered measure. Then, for every $\nu\in \ProbTwo(\X)$, if $\OptPlans(\mu,\nu)\neq \emptyset$, there exists a unique optimal transport map for the problem \eqref{MongeFormulation} with $c(x,y)=\di^2(x,y)$. 
\end{theorem}

Moreover, we prove that every probability measure, which is absolutely continuous with respect to a locally doubling measure, is densely scattered. Then, as a corollary of Theorem \ref{thm:main2}, we obtain a more classical result (Corollary \ref{cor:final}), where existence and uniqueness of optimal transport maps is proved whenever the space is locally-uniformly non-branching and the first marginal is absolutely continuous with respect to a locally doubling reference measure.

We point out that another \emph{metric Brenier theorem} was obtained by Ambrosio and Rajala in \cite{AmbRaj} under the assumption of so-called strong non-branching on the base space, and non-branching on its tangents.

The strategy we use to prove our main theorems relies on the notion of \emph{cyclical monotonity}, which is one of the most basic and important tools in the theory of optimal transport. We remark that, unlike what happens in many of the articles referred above, our argument only uses this simple notion, without appealing to the Kantorovich duality (see \cite[Lecture 3]{ABS}). Similar strategies have instead been developed in \cite{Rajala, magnabosco3,LMX}. 

We recall that a set $\Gamma\subset X\times Y$ is said to be $c$-cyclically monotone if 
\begin{equation*}
    \sum_{i=1}^{N} c\left(x_{i}, y_{\sigma(i)}\right) \geq \sum_{i=1}^{N} c\left(x_{i}, y_{i}\right)
\end{equation*}
for every $N\geq1$, every permutation $\sigma$ of $\{1,\dots,N\}$ and every $(x_i,y_i)\in \Gamma$ for $i=1,\dots,N$. The next proposition shows the strict relation  between optimality and $c$-cyclical monotonicity and will be a fundamental tool in this work.

\begin{prop}\label{prop:cmonotonicity}
Let $X$ and $Y$ be Polish spaces and $c:X\times Y \to [0,\infty]$ a lower semicontinuous cost function. Then, every optimal transport plan $\pi\in \OptPlans(\mu,\nu)$ such that $\int c \de \pi<\infty$ is concentrated on a $c$-cyclically monotone set. 
\end{prop}

\subsection*{Acknowledgments} 
S.L. acknowledges support from the Mathematical Institute (University of Oxford) through the departmental funding for summer research projects. M.M. acknowledges support from the Royal Society through the Newton International Fellowship (award number: NIF$\backslash$R1$\backslash$231659).
The authors would like to thank Tapio Rajala for valuable comments on the paper.

\section{Proof of Theorem \ref{thm:main1}}\label{sec:metrictwist}

In this section we introduce the local metric twist condition for a cost function on a metric measure space $(\X,\di, \m)$ and prove our first main result.

\begin{definition}[see Figure \ref{img: def:LMTC}]\label{def:LMTC}
In a metric measure space $(\X,\di, \m)$, we say that a cost function $c:\X\times \X\to [0,\infty]$ satisfies the Local Metric Twist Condition, in short \textbf{[LMTC]}, if for every triple $x,y,z\in \X$ with $y\neq z$, there exist $r_1,r_2>0$ such that for every $\bar{x} \in B_{r_1}(x)$,
\begin{equation*}
   \liminf _{s \to 0} \frac{\m\left(F_{s \mid r_2}(\bar{x} \mid y, z)\right)}{\m\left(B_s(\bar{x})\right)}>0,
\end{equation*}
where 
\begin{equation*}
\begin{split}
    F_{s \mid r}(x \mid y, z):=\big\{x^{\prime} \in B_s(x): c\left(x^{\prime}, y^{\prime}\right)+c\left(x, z^{\prime}\right)<c\left(x, y^{\prime}\right)&+c\left(x^{\prime}, z^{\prime}\right), \\
    &\forall y^{\prime} \in B_r(y), z^{\prime} \in B_r(z)\big\}.
    \end{split}
\end{equation*}
\end{definition}

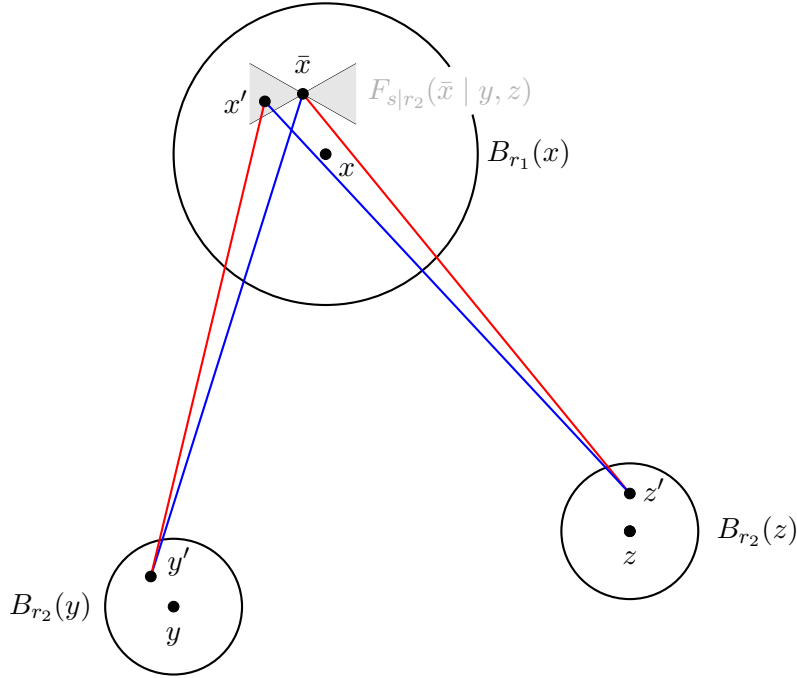
\begin{figure}[h]
    \centering
    \begin{tikzpicture}
    \filldraw (-3,-3) circle (2pt);
    \draw[thick](-3,-3) circle (0.9);
    \filldraw (3,-2) circle (2pt);
    \draw[thick](3,-2) circle (0.9);
    \filldraw (-1,3) circle (2pt);
    \draw[thick](-1,3) circle (2);
    \node at (-3,-3)[label=south:$y$] {};
    \node at (3,-2)[label=south:$z$] {};
    \node at (-1.1,2.8)[label=east:$x$] {};
    \node at (-1.3,3.8)[label=north:$\bar x$] {};
    \draw(-2,3.4)--(-0.6,4.2);
    \draw (-0.6,3.4)--(-2,4.2);
    \fill[fill=black!10!white](-2,3.4)--(-0.6,4.2)--(-0.6,3.4)--(-2,4.2)--cycle;
    \filldraw (3,-2) circle (2pt);
    \draw[blue,thick] (-3.3,-2.6)--(-1.3,3.8);
    \draw[thick, red] (3,-1.5) -- (-1.3,3.8);
    \draw[thick,blue] (3,-1.5) -- (-1.8,3.7);
    \draw[thick,red] (-3.3,-2.6) -- (-1.8,3.7);
    \filldraw (-3.3,-2.6) circle (2pt);
    \filldraw (3,-1.5) circle (2pt);
    \filldraw (-1.8,3.7) circle (2pt);
    \fill[white](0.8,3.85) circle (10pt);
    \filldraw (-1.3,3.8) circle (2pt);
    \node at (-3.35,-2.4)[label=east:$y'$] {};
    \node at (2.9,-1.45)[label=east:$z'$] {};
    \node at (-1.75,3.7)[label=west:$x'$] {};
     \node at (2.5,3)[label=west:$B_{r_1}(x)$] {};
      \node at (-3.8,-3)[label=west:$B_{r_2}(y)$] {};
       \node at (5.5,-2)[label=west:$B_{r_2}(z)$] {};
        \node at (2,3.8)[label=west:$\color{black!30!white}{F_{s \mid r_2}{(\bar x \mid y, z)}}$] {};

\end{tikzpicture}
    \caption{Definition of \textbf{[LMTC]} (Local Metric Twist Condition): as $ x' \in F_{s|r_2}(\bar x | y,z) $, the sum of the distances between points connected by a red line is smaller than the sum of the distances between points connected by a blue line, for every choice of $ y' \in B_{r_2}(y)$ and $z' \in B_{r_2}(z) $.}
    \label{img: def:LMTC}
    
\end{figure}

\begin{prop}\label{thm:PMTC(Diff) -> LMTC} 
In the metric measure space $(\X, \di, \m) = (\R^n, d_{\mathrm{Euc}}, \Leb^n)$, if a continuously differentiable (i.e. $C^1$) cost function $c:\R^n \times \R^n \to \R^+$ satisfies the twist condition \eqref{eq:twist} then it also satisfies \textbf{[LMTC]}.
\end{prop}

\begin{proof}
Fix $x,y,z\in \X$ with $y\neq z$. Since $c$ satisfies the twist condition, $w:=\nabla_x c(x, y)-\nabla_x c(x, z) \neq 0$. Now, as $c$ is $C^1$, the map $(x, y, z) \mapsto \nabla_x c(x, y)-\nabla_x c(x, z)$ is continuous, in particular there exist $r, r_2 > 0$ such that for all $ \bar{x} \in B_{r}(x), y^{\prime} \in B_{r_2}(y), z^{\prime} \in B_{r_2}(z)$ (in terms of Euclidean distance), we have $$\left\|\nabla_x c\left(\bar{x}, y^{\prime}\right)-\nabla_x c\left(\bar{x}, z^{\prime}\right)-w\right\|<\frac{\|w\|}{10 n^2}.$$ 
Therefore there exist $\delta > 0$ and $V \subseteq S^{n-1}$  such that the cone $ C_V:=\{\lambda v, 0 \leq \lambda \leq 1, v \in V\}$ spanned by $V$ has positive fraction of volume $\frac{\Leb^n\left(C_V\right)}{\Leb^n\left(B_1(0)\right)}:=L>0$ and $$\left\langle\nabla_x c\left(\bar{x}, y^{\prime}\right)-\nabla_x c\left(\bar{x}, z^{\prime}\right), v\right\rangle<-\delta<0, $$for all $\bar{x} \in B_{r_1}(x),  y^{\prime} \in B_{r_2}(y),  z^{\prime} \in B_{r_2}(z), v\in V$.

Taking $r_1 = r/2$, we claim that for all $\bar{x} \in B_{r_1}(x)$, $$\liminf _{s \to 0} \frac{\Leb^n\left(F_{s \mid r_2}(\bar{x} \mid y, z)\right)}{\Leb^n \left(B_s(\bar{x})\right)}>0, $$ this shall conclude the theorem. 
Indeed, fix any $\bar{x} \in B_{r_1}(x) = B_{r/2}(x)$, then for all $v\in V$ and $t < r/2$, $\bar x + tv \in B_r(x)$. Since $c$ is continuously differentiable, we deduce that for every $s<r/2$ 
\begin{equation*}
    \begin{aligned}
 c\left(\bar{x}+s v, y^{\prime}\right)&-c\left(\bar{x}+s v, z^{\prime}\right)-c\left(\bar{x}, y^{\prime}\right)+c\left(\bar{x}, z^{\prime}\right) \\
& =\int_0^s\left\langle\nabla_x\left(\bar{x}+t v, y^{\prime}\right)-\nabla_x\left(\bar{x}+t v, z^{\prime}\right), v\right\rangle \mathrm{d} t <\int_0^s-\delta \, \mathrm{d}t =-s \delta<0,
\end{aligned}
\end{equation*}
for all $y^{\prime} \in B_{r_2}(y), z^{\prime} \in B_{r_2}(z), v \in V$. In particular, we deduce that for each $s<r/2, v \in V$, $\bar{x}+s v \in F_{s \mid r_2}(\bar{x} \mid y, z)$. Hence, for all $s<r/2$, we have $$\frac{\Leb^n \left(F_{s \mid r_2}(\bar{x} \mid y, z)\right)}{\Leb^n \left(B_s(\bar{x})\right)} \geq \frac{\Leb^n \left(C_V\right)}{\Leb^n \left(B_1(0)\right)}=L>0,$$ and the result follows by passing to the $\liminf$ as $s\to 0$.
\end{proof}

Proposition \ref{thm:PMTC(Diff) -> LMTC} shows that the notion of \textbf{[LMTC]} is a non-vacuous and consistent extension of the Riemannian twist condition \eqref{eq:twist} to over more general metric measure spaces. Moreover, notice that the constant $L$ in proof of Proposition \ref{thm:PMTC(Diff) -> LMTC} can be chosen uniformly in the neighborhood $B_{r_1}(x)$, this shows that, in principle, \textbf{[LMTC]} is a weaker requirement than the classical twist condition. Nevertheless, \textbf{[LMTC]} is already sufficient for claiming the existence of a unique optimal transport map, as will be proved in Theorem \ref{thm: LMTC -> !OTM}.

We now introduce a pointwise and thus weaker version of \textbf{[LMTC]} and a regularity condition which, when combined together, imply \textbf{[LMTC]}.

\begin{definition}\label{def:PMTC(Met)}
In a metric measure space $(\X,\di, \m)$, we say that a cost function $c:\X\times \X \to [0,\infty]$ satisfies the Pointwise Metric Twist Condition, in short \textbf{[PMTC]}, if for every triple $x,y,z\in \X$ with $y\neq z$, there exists $\varepsilon>0$ such that 
\begin{equation*}
   \liminf _{s \to 0} \frac{\mathfrak{m}\left(E_s^{\varepsilon}({x} \mid y, z)\right)}{\mathfrak{m}\left(B_s({x})\right)}>0
\end{equation*}
where
$$E_s^{\varepsilon}(x \mid y, z):=\left\{x^{\prime} \in B_s(x): c\left(x^{\prime}, y\right)+c(x, z)<c(x, y)+c\left(x^{\prime}, z\right)-\varepsilon d\left(x^{\prime}, x\right)\right\}.$$
\end{definition}

\begin{remark}\label{rmk:PMTC(Met)}
It is helpful to view the inequality in the definition of $E_s^{\varepsilon}(x \mid y, z)$ as an estimate for the difference in ``gradients" $$\frac{c(x',y)-c(x,y)}{\di(x',x)} <  \frac{c(x',z)-c(x,z)}{\di(x',x)} - \varepsilon. $$ From this perspective, \textbf{[PMTC]} resembles the twist condition \eqref{eq:twist} and essentially requires the ``gradients" with respect to the $x$ coordinate of $c$ at $(x,y)$ and at $(x,z)$ to differ by some tolerance level $\varepsilon$. Since gradients are not readily-defined on generic metric measure spaces, this requirement is phrased in the sense that the set of $x'$ satisfying the above inequality exist densely around $x$.
\end{remark}

\begin{definition}\label{def:C1(Met)}
In a metric measure space $(\X,\di, \m)$, a cost $c:\X\times \X\to \R^+$ is \textbf{[C1]} if 
\begin{itemize}
    \item for every pair $y,z\in \X$, and for every $\varepsilon>0$, the function 
\begin{equation*}  x \mapsto 
\liminf _{s \to 0} \frac{\mathfrak{m}\left(E_s^{\varepsilon}(x \mid y, z)\right)}{\mathfrak{m}\left(B_s(x) \right)}
\end{equation*}
is continuous.

\item  for every pair $x,y\in \X$ and $\varepsilon>0$, there exists radii $r_1,r_2>0$ and $\delta>0$ such that for all $\bar{x}\in B_{r_1}(x), \bar{y} \in B_{r_2}(y)$ we have
\begin{equation*} 
    \left|\left(c\left(x^{\prime}, \bar{y}\right)-c(\bar{x}, \bar{y})\right)-\left(c\left(x^{\prime}, y\right)-c(\bar{x}, y)\right)\right|<\varepsilon \di\left(x^{\prime}, \bar{x}\right), \quad \forall x'\in B_\delta(\bar x).
\end{equation*}

\end{itemize}

\end{definition}
\begin{remark}
    The first requirement in Definition \ref{def:C1(Met)}, when viewed as in Remark \ref{rmk:PMTC(Met)}, imposes a uniform tolerance level $\varepsilon$ on the ``description" of gradients on a neighbourhood of $x$, virtually asking for continuity of $c$'s ``gradient" in its first argument. By a similar manner, the second condition additionally claims the continuity of $c$'s ``gradient" with respect to its both arguments. Indeed, when $c: \R^n \times \R^n \to \R^+$ is continuously differentiable, the first condition of \textbf{[C1]} is readily checked, while the second condition essentially asks: for all $x,y\in \R^N $ there exists $ r_1,r_2>0$ such that for all $\bar{y} \in B_{r_2}(y), \bar{x} \in B_{r_1}(x)$, there is $\left|\nabla_x c(\bar{x}, \bar{y})-\nabla_x c(\bar{x}, y)\right|<\varepsilon$, which is also immediate from the continuity of $c$'s gradient.
\end{remark}

\begin{prop}\label{thm:PMTC(Met) + C1(Met) = LMTC}
    On a metric measure space $(\X, \di, \m)$, if cost function $c:\X \times \X \to \R^+$ satisfies \textbf{[PMTC]} and is \textbf{[C1]}, then it satisfies \textbf{[LMTC]}. 
\end{prop}
\begin{proof}
    We need to prove the statement: for any triplet $x,y,z$ there exists $r_1,r_2>0$ so that, 
    \begin{equation}\label{eq:uniform positivity}
    \forall \bar{x} \in B_{r_1}(x), \quad \quad \liminf _{s \to 0} \frac{\mathfrak{m}\left(F_{s \mid r_2}(\bar{x} \mid y, z)\right)}{\mathfrak{m}\left(B_s(\bar{x})\right)}>0. 
    \end{equation}
    To this aim, fix $x,y,z \in \X$. Since $c$ satisfies \textbf{[PMTC]}, there exist some $\varepsilon > 0$ such that $\liminf _{s \rightarrow 0} \frac{\mathfrak{m}\left(E_s^{\varepsilon}(x \mid y, z)\right)}{\mathfrak{m}\left(B_s(x)\right)}>0$. 
    By continuity of $x \mapsto \liminf _{s \rightarrow 0} \frac{\mathfrak{m}\left(E_s^{\varepsilon}(x \mid y, z)\right)}{\mathfrak{m}\left(B_s(x)\right.}$, there exist $r_0>0$ so that for all $\bar{x} \in B_{r_0}(x)$, $$\liminf _{s \to 0} \frac{\mathfrak{m}\left(E_s^{\varepsilon}(\bar{x} \mid y, z)\right)}{\mathfrak{m}\left(B_s(\bar{x})\right)}>0.$$

    By the second condition of \textbf{[C1]}, there exist some $\delta_1>0$ and some radii $r_{11}, r_{12}>0$ such that for all $\bar{x}\in B_{r_{11}}(x), \bar{y} \in B_{r_{12}}(y)$, as long as $\di(x', \bar{x} )< \delta_1$, 
    \begin{equation} \label{eq:xy-twist}
    \left|\left(c\left(x^{\prime}, \bar{y}\right)+c(\bar{x}, y)\right)-\left(c\left(x^{\prime}, y\right)+c(\bar{x}, \bar{y})\right)\right|<\frac{\varepsilon}{3} \di \left(x^{\prime}, \bar{x}\right),
    \end{equation}
  Similarly, there exist some $\delta_2>0$ and some radii $r_{21}, r_{22}>0$ such that for all $\bar{x}\in B_{r_{21}}(x), \bar{z} \in B_{r_{22}}(z)$, as long as $\di (x', \bar{x} )< \delta_2$, 
\begin{equation}\label{eq:xz-twist}
     \left|\left(c\left(x^{\prime}, \bar{z}\right)+c(\bar{x}, z)\right)-\left(c\left(x^{\prime}, z\right)+c(\bar{x}, \bar{z})\right)\right|<\frac{\varepsilon}{3} \di \left(x^{\prime}, \bar{x}\right).
\end{equation}
Take $r_1:=r_{11} \wedge r_{21} \wedge r_0$, $r_2=r_{12} \wedge r_{22}$ and $\delta:= \delta_1 \wedge \delta_2$, then both \eqref{eq:xy-twist} and \eqref{eq:xz-twist} hold for all $\bar{x} \in B_{r_1}(x), \bar{y} \in B_{r_2}(y), \bar{z} \in B_{r_2}(z)$, whenever $\di (x', \bar{x} )< \delta$. It remains to validate statement \eqref{eq:uniform positivity} with these selections of $r_1$ and $r_2$. 

Fix an arbitrary $\bar{x}\in B_{r_1}(x)$, then $\bar{x}\in B_{r_0}(x)$, and thus $\liminf _{s \to 0} \frac{\mathfrak{m}\left(E_s^{\varepsilon}(\bar{x} \mid y, z)\right)}{\mathfrak{m}\left(B_s(\bar{x})\right)}>0$.
For any $s< \delta$ and $x^{\prime} \in E_s^{\varepsilon}(\bar{x} \mid y, z)$, for all $\bar{y} \in B_{r_2}(y), \bar{z} \in B_{r_2}(z)$, we have 

$$\begin{aligned}
 c\left(x^{\prime}, \bar{y}\right)+c(\bar{x}, \bar{z})-c(\bar{x}, \bar{y})&-c\left(x^{\prime}, \bar{z}\right) \\
&= \; {\left[\left(c\left(x^{\prime}, \bar{y}\right)+c(\bar{x}, y)\right)-\left(c\left(x^{\prime}, y\right)+c(\bar{x}, \bar{y})\right)\right] } \\
&\quad  -\left[\left(c\left(x^{\prime}, \bar{z}\right)+c(\bar{x}, z)\right)-\left(c\left(x^{\prime}, z\right)+c(\bar{x}, \bar{z})\right)\right] \\
&\quad  -\left[\left(c(\bar{x}, y)+c\left(x^{\prime}, z\right)\right)-\left(c\left(x^{\prime}, y\right)+c(\bar{x}, z)\right)\right] \\
&:=   \mathrm{Cost}_{xy} - \mathrm{Cost}_{xz} - \mathrm{Cost}_{xyz} <  0,
\end{aligned}$$
because $|\mathrm{Cost}_{xy}|,|\mathrm{Cost}_{xz}|<\frac{\varepsilon}{3} \di \left(x^{\prime}, \bar{x}\right)$ and $\mathrm{Cost}_{xyz}>\varepsilon \di \left(x^{\prime}, \bar{x}\right)$. Therefore, we have just shown that $x^{\prime} \in F_{s \mid r_2}(\bar{x} \mid y, z)$.
That is to say, whenever $s<\delta$, $E_s^{\varepsilon}(\bar{x} \mid y, z) \subseteq F_{s \mid r_2}(\bar{x} \mid y, z)$, and hence 
$$\liminf _{s \to 0} \frac{\mathfrak{m}\left(F_{s \mid r_2}(\bar{x} \mid y, z)\right)}{\mathfrak{m}\left(B_s(\bar{x})\right)} \geq \liminf _{s \to 0} \frac{\mathfrak{m}\left(E_s^{\varepsilon}(\bar{x} \mid y, z)\right)}{\mathfrak{m}\left(B_s(\bar{x})\right)}>0.$$
\end{proof}

Before proceeding to prove that \textbf{[LMTC]} implies the existence of unique transport maps in Theorem \ref{thm: LMTC -> !OTM}, it is helpful to prove two lemmas of use.

\begin{lemma}\label{lem:NotConcDiag}
    Suppose $\pi\in \OptPlans(\mu, \nu)$ is not induced by a map and let $(\pi_x)$ be the disintegration kernel of $\pi$ via projection onto the the first component (that is $\pi = \int \pi_x \,\mathrm{d}\mu(x)$). Define $\eta=\int \pi_x \times \pi_x d \mu(x)$. Then $\eta$ is a probability measure not concentrated on the diagonal $\Delta:=\{(x, x), x \in \X\}$.
\end{lemma}
\begin{proof}
By standard arguments, it is readily checked that $\eta$ is well-defined.
As $\pi$ is not induced by a transport map, there exist $A\subset \X$ with $\mu(A)>0$ such that for each $x\in A$, $\pi_x$ is not a $\delta$-measure and hence $\pi_x \times \pi_x(\Delta)<1$. It is easy to check that $\eta(\X \times \X) = 1$, while
$$\begin{aligned} \eta(\Delta):=
 \int_A \pi_x \times \pi_x(\Delta) \de \mu(x)+\int_{A^c} \pi_x \times \pi_x(\Delta) \de \mu(x)
< \mu(A)+\mu\left(A^c\right) 
= 1
\end{aligned},$$
showing that $\eta$ is not concentrated on $\Delta$.
\end{proof}

\begin{lemma}\label{lem:RestrictOptimal}
Let $c:\X\times \X\rightarrow \R^+$ be a lower-semicontinuous cost, suppose $\mu \ll \m$ and let $\pi \in \OptPlans(\mu, \nu)$. Then, the following hold.
\begin{itemize}

\item Fix any $\hat{x}\in \X$ and $r>0$ such that $\pi(\bar{B}_r(\hat{x}) \times \bar{B}_r(\hat{x}))>0$. Let  $\hat{\pi}:=\frac{\pi \llcorner (\bar{B}_r(\hat{x}) \times \bar{B}_r(\hat{x}))}{\pi(\bar{B}_r(\hat{x}) \times \bar{B}_r(\hat{x}))}$ and let $
 \hat{\mu}:={\p_1}_\# \hat{\pi}, 
\hat{v}:={\p_2}_\# \hat{\pi}
$ be the pushforward measures by the coordinate projections. Then $\hat{\pi} \in \OptPlans(\hat{\mu}, \hat{\nu})$.

\item If $\pi$ is not induced by a map, then one can find a $\hat{x} \in \X$ and $r>0$, so that the restriction $\hat{\pi}$ is not induced by a map.

\end{itemize}
\end{lemma}
\begin{proof}

For the first claim, assume $\hat{\pi} \notin \OptPlans(\hat{\mu}, \hat{\nu})$. Then there exist $\tilde{\pi} \in \mathsf{Adm}(\hat{\mu}, \hat{\nu})$ non-negative and with the same marginals as $\hat{\pi}$, but with $\int c \de \tilde{\pi}<\int c \de \hat{\pi}$. Define $\pi^*:=\pi+ \pi(\bar{B}_r(\hat{x}) \times \bar{B}_r(\hat{x}))( \tilde{\pi}- \hat{\pi})$. Now, since $\hat{\pi}$ and $\tilde{\pi}$ have the same marginals, we have that 
$$\begin{aligned}
{\p_1} _\# \pi^*={\p_1}_\#\pi=\mu \quad \text{and} \quad {\p_2} _\# \pi^*={\p_2}_\#\pi=\nu,
\end{aligned}$$
and therefore $\pi^* \in \mathsf{Adm}(\mu, \nu)$. However, 
$$\begin{aligned}
\int c \de \pi^* =\int c \de \pi  + \pi(\bar{B}_r(\hat{x}) \times \bar{B}_r(\hat{x})) \int c \de(\tilde\pi-\hat{\pi}) <\int c \de \pi,
\end{aligned}$$
which contradicts the optimality of $\pi$.

For the second claim, call $\widehat{\pi}_r:=\pi \llcorner\left(\overline{B_r}(\hat{x}) \times \overline{B_r}(\hat{x})\right)$. Let $A:=\left\{x: \pi_x \text { not a } \delta \text {-measure}\right\}$ and  $A_r:=\left\{x \in \overline{B_r}(\hat{x}):\left(\widehat{\pi_r}\right)_x \text { not a } \delta \text {-measure}\right\}$. Noticing that $A_r\nearrow A  \text{ and } \mu(A)>0$ allows to prove the claim.
\end{proof}

Lemma \ref{lem:RestrictOptimal} allows us to restrict attention to optimal plans with compact support in proper spaces. In particular, assuming a plan is not induced by a map, one may consider the restriction of the optimal plan to some compact subset of choice, which remains an optimal plan, and derive a contradiction for the restricted plan.

\begin{theorem}\label{thm: LMTC -> !OTM}
Let $(\X,\di, \m)$ be a proper metric measure space with $\m$ doubling on $\X$. Let $c:\X\times \X\to [0,\infty]$ be any lower-semi-continuous cost function that satisfies \textbf{[LMTC]}. Then for every pair of measures $\mu,\nu\in \Prob(\X)$ with $\mu \ll \m$ and such that $\OptPlans(\mu,\nu) \neq \emptyset $,\footnote{This is just to avoid the marginal (trivial) case of having no admissable plans of finite cost.} there exists a unique optimal transport plan $\pi\in\OptPlans(\mu,\nu)$ which is induced by a map.
\end{theorem}
\begin{proof}
Once it is proved that every optimal plan is induced by a map, it can be concluded that the optimal plan is unique. Indeed, assume by contradiction that $\pi_1$ and $\pi_2$ are different optimal plans induced by $\mu$-different measurable maps $T_1,T_2 : X \to X$. That is, $\pi_1 = (\mathrm{id}\times T_1)_{\#} \mu$, $\pi_2 = ( \mathrm{id} \times T_2)_{\#} \mu$, and also $\mu({T_1 \neq T_2})>0$. Define
\begin{equation*}
    \hat{\pi}:=\frac{1}{2} \pi_1+\frac{1}{2} \pi_2=\left(\frac{\mathrm{id}\times T_1+\mathrm{id}\times T_2}{2}\right)_{\#} \mu,
\end{equation*}
then by convexity $\hat{\pi}$ remains an optimal plan, but $\hat{\pi}$ can not be induced by a map.

Therefore, we fix a $\pi \in \OptPlans(\mu, \nu)$ and we want to show $\pi$ is induced by a map.
Assume by contradiction that $\pi$ is not induced by a map. Since $\X$ is proper, by Lemma \ref{lem:RestrictOptimal} we can assume without loss of generality that $\pi$ has (compact) support contained in $\overline{B_R}(\hat{x})\times \overline{B_R}(\hat{x})$, and therefore that $\supp(\mu),\supp(\nu) \subset \overline{B_R}(\hat{x})$. Then, applying Proposition \ref{prop:cmonotonicity}, we deduce that $\pi$ is concentrated on a $c$-cyclical monotonic set $\Gamma \subseteq \overline{B_R}(\hat{x}) \times \overline{B_R}(\hat{x})$. Moreover, let $(\pi_x)_{x\in \X}$ be the disintegration kernel of $\pi$ via projection onto the first component, then $\pi_x(\Gamma_x) = 1$ for $\mu$-a.e. $x\in \X$, where $\Gamma_x:=\{z\in \X: (x,z)\in \Gamma\}$. 

Defining $\eta=\int \pi_x \times \pi_x d \mu(x)$ as in Lemma \ref{lem:NotConcDiag}, then $\eta$ is not concentrated on the diagonal $\Delta$, hence we can find $(y, z) \in \operatorname{supp}(\eta)$ with $y \neq z$.
For all $r>0$, define $$E_r:=\left\{x \in X: \pi_x\left(B_r(y) \cap \Gamma_x\right), \pi_x\left(B_r(z) \cap \Gamma_x\right)>0\right\}.$$ 
Then for all $r>0$, $$\eta\left(B_r(y) \times B_r(z)\right)=\int_{E_r} \pi_x\left(B_r(y)\right) \pi_x\left(B_r(z)\right) d \mu(x)>0,$$ showing $\mu\left(E_r\right)>0$ for all positive $r$.

For each $n\in \mathbb{N}$, by the absolute continuity of $\mu$ with respect to $\m$, $\mu(E_{1/n})>0$ implies $\m(E_{1/n})>0$. Therefore, one can take $x_n\in E_{1/n}$ to be an $\m$-density point of $E_{1/n}$. 
Since $\pi$ is concentrated on $\Gamma \subseteq \overline{B_R}(\hat{x}) \times \overline{B_R}(\hat{x})$, for every $n$ $E_{1 / n} \subseteq \overline{B_R}(\hat{x})$ which is compact, and hence there exists a converging subsequence $x_{n_k} \rightarrow x \in \overline{B_R}(\hat{x}) \subseteq \X$. Relabel $(x_{n_k})$ as $(x_n)$ without loss of generality. 
In particular,  for all $\varepsilon>0$ and for all $N \in \mathbb{N}$, there exist an $ n>N$ such that $ \di\left(x_n, x\right)<\varepsilon$ and also $E_{1 / n} \subseteq E_{1 / N} \text { since } n>N$ and $E_r$ are nested decreasing sets.

Now consider the triplet $(x,y,z)$. Since $c$ satisfies \textbf{[LMTC]}, there exist $r_1,r_2>0$ such that for all $ \bar{x} \in B_{r_1}(x)$, $$\liminf _{s \to 0} \frac{\m\left(F_{s \mid r_2}(\bar{x} \mid y, z)\right)}{\m\left(B_s(\bar{x})\right)}>0.$$ 
Since $F_{s \mid r_2}(\bar{x} \mid y, z)$ grows as $r_2$ decreases, one can assume $r_2$ is small enough so that $B_{r_2}(y)\cap B_{r_2}(z)=\emptyset$. Take $\varepsilon = r_1/2$ and $N > 1/r_2$, then by the above, there exist $n>N$ such that $\di(x_n,x) < \varepsilon$ and with $x_n \in E_{1/n}$ being a density point of $E_{1/n}$. Since $r_2 > 1/n$, $E_{1/n} \subset E_{r_2}$ so $x_n \in E_{r_2}$ is a density point of $E_{r_2}$ as well. Rename $\bar{x}:= x_n$.

Then, since $\di(\bar{x}, x) < \varepsilon < r_1$, $$\liminf _{s \to 0} \frac{\m\left(F_{s \mid r_2}(\bar{x} \mid y, z)\right)}{\m\left(B_s(\bar{x})\right)}=\delta,$$ for some $\delta > 0$. While, as $\bar{x}$ is a density point of $E_{r_2}$,
$$\liminf _{s \to 0} \frac{\m\left(B_s(\bar{x}) \cap E_{r_2}\right)}{\m\left(B_s(\bar{x})\right)}=1.$$ Hence for $s$ small enough so that the sum of these densities exceeds $1$, exists $\tilde{x} \in F_{s \mid r_2}(\bar{x} \mid y, z) \cap B_s(\bar{x}) \cap E_{r_2}$. Finally, since $\bar{x}, \tilde{x} \in E_{r_2}$, there exists $\bar{y} \in B_{r_2}(y)$, $\bar{z}\in B_{r_2}(z)$ with $(\bar x,\bar y),(\tilde x,\bar z)\in \Gamma$ and $\bar{y} \neq \bar{z}$. Also, since $\tilde{x} \in F_{s \mid r_2}(\bar{x} \mid y, z)$, for all $y^{\prime} \in B_{r_2}(y), z^{\prime} \in B_{r_2}(z),$ $c\left(\tilde{x}, y^{\prime}\right)+c\left(\bar{x}, z^{\prime}\right)<c\left(\bar{x}, y^{\prime}\right)+c\left(\tilde{x}, z^{\prime}\right) $. In particular, $c(\tilde{x}, \bar{y})+c(\bar{x}, \bar{z})<c(\bar{x}, \bar{y})+c(\tilde{x}, \bar{z})$ contradicting the $c$-cyclical monotonicity of $\Gamma$.

\end{proof}

\section{Proof of Theorem \ref{thm:main2}}\label{sec:main2}

In this section we prove our second main result. We start by introducing the notion of geodesic cone and densely scattered measure in a metric space $(\X,\di)$. In the next definition, $\Geo(\X,\di)$ denotes the set of geodesics (i.e. length minimising curves) in $(\X,\di)$, parameterized with constant speed on $[0,1]$.

\begin{definition}[see Figure \ref{def:cone}]\label{def:cone}
For every geodesic $\gamma\in \Geo(\X,\di)$ and $k>0$, we define the geodesic cone of size $k$ around $\gamma$ as 
\begin{equation*}
    C_{\gamma,k}:= \bigcup_{t\in [0,1]} \overline B_{tk}(\gamma(t)).
\end{equation*}
\end{definition}

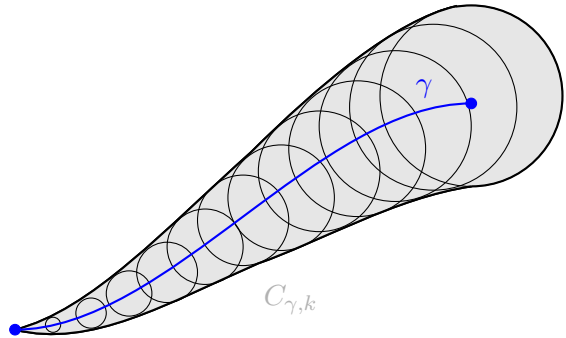
\begin{figure}[h]

    \centering
\begin{tikzpicture}
    \filldraw[fill=black!10!white, thick] (-3,0)..controls (-2,0.25) and (-1,1.36).. (0.1,2.35)..controls (1,3.25) and (2,4.1).. (2.8,4.29)-- (3,3) --(3,1.9) ..controls (2,1.75) and (1,1.16) .. (0.2,0.88) ..controls (-1,0.35) and (-2,-0.25)..(-3,0);
    \filldraw[fill=black!10!white,thick](3,3.1) circle (1.2);
    \fill[fill=black!10!white] (-3,0)..controls (-2,0.25) and (-1,1.36).. (0.1,2.35)..controls (1,3.25) and (2,4.1).. (2.8,4.29)-- (3,3) --(3,1.9) ..controls (2,1.75) and (1,1.16) .. (0.2,0.88) ..controls (-1,0.35) and (-2,-0.25)..(-3,0);
    \draw[thick] (-3,0)..controls (-2,0.25) and (-1,1.36).. (0.1,2.35)..controls (1,3.25) and (2,4.1).. (2.8,4.29);
    \draw[thick] (3,1.9) ..controls (2,1.75) and (1,1.16) .. (0.2,0.88) ..controls (-1,0.35) and (-2,-0.25)..(-3,0);
    \draw(3,3.1) circle (1.2);
    \draw(-2.5,0.065) circle (0.10);
    \draw(-2,0.22) circle (0.2);
    \draw(-1.5,0.48) circle (0.3);
    \draw(-1,0.76) circle (0.4);
    \draw(-0.5,1.1) circle (0.5);
    \draw(0,1.45) circle (0.6);
    \draw(0.5,1.75) circle (0.7);
    \draw(1,2.1) circle (0.8);
    \draw(1.5,2.4) circle (0.9);
    \draw(2,2.7) circle (1);
    \draw(2.5,2.95) circle (1.1);
    \draw[thick, blue] (-3,0)..controls (-1,0) and (1,3).. (3,3);
    \node at (2,3.2)[label=east:$\color{blue}\gamma$] {};
    \node at (0,0.4)[label=east:$\color{black!30!white}C_{\gamma,k}$] {};
    \filldraw[blue] (-3,0) circle (2pt);
    \filldraw[blue] (3,3) circle (2pt);
\end{tikzpicture}
    \caption{Definition of $C_{\gamma,k}$}
    \label{img: def:cone}
\end{figure}

\begin{definition}\label{def:denselyscattered}
We say that a probability measure $\mu\in \Prob(\X)$ is densely scattered if it is doubling on $(\supp(\mu),\di)$ and it asymptotically gives mass to every geodesic cone, that is, for $\mu$-almost every $x\in \X$, every geodesic $\gamma$ starting in $x$ ($\gamma(0)=x$) and every size $k>0$, 
\begin{equation*}\label{eq:densescat}
    \liminf_{r\to 0} \frac{\mu(C_{\gamma,k} \cap B_r(x))}{\mu(B_r(x))} >0.
\end{equation*}
\end{definition}

\begin{prop}\label{prop:doublingtoDS}
    Suppose there exists a reference measure $\m$ which is locally doubling on the whole space $(\X,\di)$. Then, every probability measure $\mu\in \Prob(\X)$ such that $\mu\ll\m$, is densely scattered. 
\end{prop}

\begin{proof}
    As $\mu\ll\m$, Radon-Nikodym theorem guarantees the existence of a $\m$-measurable (nonnegative) function $f$ such that $\mu=f\m$. Observe that $\mu$-almost every $x\in \X$ belongs to and is a density point of a set of the form 
    \begin{equation*}
        S_a^b := \{ y\in \X \mid f(y)\in (a,b)\}, \qquad \text{for some }a,b>0.
    \end{equation*}
    Therefore, it is sufficient to prove \eqref{eq:densescat} for $x\in S_a^b$ being a density point of $S_a^b$ (for some $a,b>0$), every geodesic $\gamma$ starting in $x$ and every size $k>0$. 
    
Observe that, for any $r$ small enough, there exists $t(r)\in[0,1]$ such that $t(r)d(\gamma_0,\gamma_1)+t(r)k=r$.
In particular, we have that $B_{t(r)k}(\gamma_{t(r)})\subset C_{\gamma,k}\cap B_r(x)$.
On the other hand, for any given $y\in B_r(x)$, we have that
\begin{align*}
    d(\gamma_{t(r)},y)< t(r)d(\gamma_0,\gamma_1)+r=t(r)[2d(\gamma_0,\gamma_1)+k]\le 2^mt(r)k,
\end{align*}
for some $m\in \mathbb N$ independent of $r$.
As a consequence, we have that
\begin{align*}
    \liminf_{r\to0} \frac{\m(B_{t(r)k}(\gamma_{t(r)}))}{\m(B_r(x))}\ge \liminf_{r\to0}\frac{\m(B_{t(r)k}(\gamma_{t(r)}))}{\m(B_{2^mt(r)k}(\gamma_{t(r)}))}\ge C>0,
\end{align*}
where the constant $C$ depends only on the doubling constant of $\m$ around $x$ and on $m$. Now, observe that, since we assumed $x\in S_a^b$ to be a density point of $S_a^b$, we know that 
\begin{equation*}
    \mu(B_r(x)) \leq 2 b\, \m(B_r(x)) \quad \text{ and }\quad \mu(B_{t(r)k}(\gamma_{t(r)}))\geq \frac{a}{2}\,\m(B_{t(r)k}(\gamma_{t(r)})),
\end{equation*}
definitely as $r\to 0$. Finally, we can deduce that
\begin{align*}
    \liminf_{r\to0} \frac{\mu(C_{\gamma,k}\cap B_r(x))}{\mu(B_r(x))}\ge \liminf_{r\to0} \frac a {4b}  \frac{\m(B_{t(r)k}(\gamma_{t(r)}))}{\m(B_{2^mt(r)k}(\gamma_{t(r)}))}\ge \frac{aC}{4b}>0,
\end{align*}
concluding the proof.
\end{proof}

\begin{remark}
    We highlight that, in the definition of densely scattered measure (Definition \ref{def:denselyscattered}), we only require the measure to be doubling on its support. In particular, the introduction of the notion of densely scattered measure, allows us to obtain our main result (Theorem \ref{thm:NonBranch+ DenselyScattered -> !OTM}), even without assuming the existence of a nice reference measure. 
\end{remark}

We now proceed to define the notion of locally-uniformly non-branching, which will be the fundamental ingredient to prove existence and uniqueness of optimal transport maps in Theorem \ref{thm:NonBranch+ DenselyScattered -> !OTM}.

\begin{definition}\label{def:LocalUniformNonBranch}
    A metric space $(\X,\di)$ is said to be locally-uniformly non-branching if it is Polish and for every triple $x,y,z\in \X$ such that $y\neq z$ and $\di(x,y)=\di(x,z)$, there exists $r>0$ and a constant $\rho>0$ with the following property: for every triple $x'\in B_{r}(x)$, $y'\in B_r(y) $ and $ z'\in B_r(z)$ there exists a geodesic $\gamma\in \Geo(\X,\di)$, $\gamma(0)=x', \gamma(1)=y'$,  so that
\begin{equation*}
    \liminf_{t\downarrow 0} \frac{ \di(\gamma(t), z')-  \di(x', z')}{t} \geq  - (1- \rho) \di(x',y') .
\end{equation*}
\end{definition}

\begin{remark} (i) For any triplet $(x',y',z')$, it always holds that $$\liminf_{t\downarrow 0} \frac{ \di(\gamma(t), z')-  \di(x', z')}{t} \geq  -\di(x',y') $$ by the triangular inequality. The constant $\rho$ in Definition \ref{def:LocalUniformNonBranch} adds a quantitative estimate on non-branching of geodesics.
 Observe that in the Euclidean space the quantity on the left hand side is equal to $-\cos\alpha \cdot \di(x',y')$ where $\alpha$ is the angle between $y'-x'$ and $z'-x'$. Therefore, by continuity, one may choose $\rho$ such that $(1-\rho)>\cos\alpha$ to obtain the desired inequality, cf. the proof of Proposition \ref{thm: C1ConvexNorm -> LocUnifNonBranch}.

(ii) Definition \ref{def:LocalUniformNonBranch} resembles a lower curvature bound of Alexandrov type. Indeed, using triangle comparison, it is possible to prove that every Alexandrov space is locally-uniformly non-branching. Moreover, the latter is a significantly weaker assumption compared to the triangle comparison definition of Alexandrov spaces, as it only requires a condition at the limit. This is demonstrated in Proposition \ref{thm: C1ConvexNorm -> LocUnifNonBranch}, where we prove that locally-uniformly non-branching holds also in metric spaces which are not Alexandrov.
\end{remark}

\begin{prop}\label{thm: C1ConvexNorm -> LocUnifNonBranch}
    Let $\| \cdot \|$ be a $C^1$ and strictly convex norm on $\R^n$, and let $\di(x,y)=\| x-y\|$. Then $(\R^n,\di)$ is locally-uniformly non-branching.
\end{prop}

\begin{proof}
Denote $g(\cdot)=\|\cdot\|$ and fix $x',y',z'\in \R^n$. Let $\gamma$ be the straight-line geodesic connecting $x'$ and $y'$, i.e. $\gamma(t)=(1-t)x' + ty'$. Then
$$\liminf _{t \to 0} \frac{\di(\gamma(t), z')-\di(\gamma(0), z')}{t}=\nabla g(x'-z') \cdot(y'-x'),$$
and the limit is actually attained. 

Now fix $x,y,z\in \R^n$, we would like to show that, for some $r>0$ and $\rho>0$, for any triplet $x'\in B_{r}(x)$, $y'\in B_r(y) $ and $ z'\in B_r(z)$ it holds that 

\begin{equation}\label{eqn:graddot>1-rho...}
\nabla g(x'-z') \cdot(y'-x') \geq-(1-\rho) \di(x', y'),
\end{equation} 
 Observing that $g$ is continuously differentiable away from $0$, and that all terms in \eqref{eqn:graddot>1-rho...} are continuous with respect to changes in $x',y',z'$ and away from $0$ on suitably small neighbourhoods of $x,y,z$, it suffices to show that $$\nabla g(x-z)  \cdot \left( {-\frac{y-x}{\di(x,y)}}\right) < 1.$$ 

Observe that $\left\|\frac{y-x}{\di(x, y)}\right\|=1$. Now, denoting with $\norm{\cdot}_*$ the dual norm of $\|\cdot\|$, we observe that
$$
\begin{aligned}
\|\nabla g(x-z)\|_*&:=\sup _{\|v\|=1}|\nabla g(x-z) \cdot v| \\
&=\sup _{\|v\|=1} \lim_{t\to 0}\frac{ |\di (x+tv, z )-\di (x, z ) |}{t} \leq \sup _{\|v\|=1} \lim_{t\to 0} \frac{\di(x+tv,x)}{t}= 1.
\end{aligned}
$$
Moreover, picking $v$ in direction of $z-x$ the inequality is instead an equality and thus we obtain $\|\nabla g(x-z)\|_* = 1$. As a consequence, we deduce that $$\nabla g(x-z)  \cdot \left( {-\frac{y-x}{\di(x,y)}}\right) \leq 1.$$

Now assume by contradiction that strict inequality is not attained, that is, $\nabla g(x-z)  \cdot \left( {-\frac{y-x}{\di(x,y)}}\right) = 1.$ Then, since $\| \cdot \|$ is strictly convex, $\nabla g(x-z)$ is just the unique dual vector of $-\frac{y-x}{\di(x, y)}$. On the other hand, since $\|\cdot \|$ is $C^1$ and strictly convex, according to \cite[Lemma 2.9]{magnaboscorossi} we obtain $\nabla g(x-z) = \frac{(x-z)^*}{\|x-z\|}$. Therefore, we conclude that $\Big(\frac{x-y}{\|x-y\|}\Big)^*=\Big(\frac{x-z}{\|x-z\|}\Big)^*$ and thus that $y = z$, finding a contradiction.
\end{proof}

\begin{theorem}\label{thm:NonBranch+ DenselyScattered -> !OTM}
Let $(\X,\di)$ be a locally-uniformly non-branching metric space, and let $\mu\in \Prob(\X)$ be a densely scattered measure. Let $c(x,y)=\di^2(x,y)$. Then, for every $\nu\in \Prob(\X)$, if $\OptPlans(\mu,\nu)\neq \emptyset$, there exists a unique optimal transport plan $\pi \in \OptPlans(\mu,\nu)$ and it is induced by a map. 
\end{theorem}

\begin{remark}
    If, in Theorem \ref{thm:NonBranch+ DenselyScattered -> !OTM}, we have that $\mu,\nu \in \ProbTwo(\X)$ (the set of probabilities in $\Prob(\X)$ with finite second order moment), then $\OptPlans(\mu,\nu)\neq \emptyset$ and there is no need to add this as an assumption.
\end{remark}

\begin{proof}
As is in proof of Theorem \ref{thm: LMTC -> !OTM}, it is sufficient to prove that every $\pi\in\OptPlans(\mu,\nu)$ is induced by an optimal transport map. 
Assume by contradiction that $\pi\in\OptPlans(\mu,\nu)$ is not induced by a map and consider $\hat\pi\in \OptPlans(\hat\mu,\hat\nu)$ as defined in Lemma \ref{lem:RestrictOptimal}. Since $\hat \pi$ has compact support and $c$ is lower-semicontinuous, then $\hat\pi$ is concentrated on a $c$-cyclically monotone compact set $\Gamma \subset \X \times \X$. Let $(\hat\pi_x)$ be the disintegration kernel of $\hat\pi$ via projection onto the first component, (i.e. $\hat\pi = \int _\X \hat\pi_x \mathrm{d}\hat\mu (x)$), then $\hat\pi_x(\Gamma_x)=1$ for $\hat\mu$-almost every $x\in \X$. Moreover, as in proof of Theorem \ref{thm: LMTC -> !OTM}, there exist points $y^*\neq z^*$ such that for every $r>0$ the set
\begin{equation*} E_r := \{x\in \X \suchthat \hat\pi_x(B_r(y^*)\cap \Gamma_x),\hat\pi_x(B_r(z^*)\cap \Gamma_x)>0   \}
\end{equation*}
has positive $\hat\mu$-measure and thus it also has positive $\mu$-measure.

For every $n$, let $x_n \in E_{1/n}$ be density points of $E_{1/n}$ (with respect to the measure $\mu$). Since $\mu(E_{1/n})>0$ for all $n$, without loss of generality we can assume all $x_n$'s to be different. Since $E_{1 / n} \subset \p_1(\supp \pi)$ which is bounded by some compact set, there exists subsequence $x_{n_k}$, which can be relabeled $x_n$ without loss of generality, that converges to some $x^* \in \X$. 

Without loss of generality, we assume that $\di\left(x^*, y^*\right) \geq \di\left(x^*, z^*\right)$. Now the proof is divided into two cases. When $\di(x^*, y^*) = \di(x^*, z^*)$, we take advantage of the locally-uniformly non-branching assumption, while when $\di\left(x^*, y^*\right) \geq \di\left(x^*, z^*\right)$ the strict convexity of the cost function will lead to a contradiction.

 \textbf{Case 1}: $\di(x^*, y^* ) = \di(x^*, z^*)$.
 
Since $(X,\di)$ is locally-uniformly non-branching, we can take $r_0,\rho>0$ be such that for all $x \in B_{r_0}\left(x^*\right), y \in B_{r_0}\left(y^*\right), z \in B_{r_0}\left(z^*\right)$ there exists geodesic $\gamma $ connecting $x$ and $y$ such that $$\liminf _{t \rightarrow 0} \frac{\di(\gamma(t), z)-\di(\gamma(0), z)}{t} \geq-(1-\rho) \di(x, y).$$

Fix some $0<\delta < \frac{\rho}{1-\rho}$. Since 
 $(1+\delta) \di \left(x^*, y^*\right)> \di \left(x^*, z^*\right)$, by continuity of the distance, there exist some $r_1 > 0 $ such that for all $y \in B_{r_1}\left(y^*\right), z \in B_{r_1}\left(z^*\right), x \in B_{r_1}\left(x^*\right)$, $(1+\delta) \di (x, y) \geq \di (x, z)$.  Let $r_2:=\frac{1}{3} \min \left\{\di\left(x^*, y^*\right), \di\left(y^*, z^*\right)\right\}$ to ensure the balls are disjoint (see Fact II below) and take $r=r_0 \wedge r_1 \wedge r_2$. Since $x_n \to x^*$, there exist $n$ such that $\di(x_n,x^*)<r$ and $1/n < r$ so that $x_n \in E_{\frac{1}{n}} \subseteq E_r$. If we relabel this $x_n$ as $x$, then the following facts can be checked:

 \begin{itemize}
      \item Fact I: $x\in E_r$ is a density point of $E_r$.
     \item Fact II: $B_r(y^*)\cap B_r(z^*) = \emptyset$.
     \item Fact III: For any $y\in B_r (y^* ),  z\in B_r (z^* )$, $$(1+\delta) \di(x, y) \geq \di(x, z).$$
     \item Fact IV: For any $y \in B_r\left(y^*\right), z \in B_r\left(z^*\right)$, there exists a geodesic $\gamma$ connecting $x$ and $y$, satisfying $$\liminf \frac{\di(\gamma(t), z)-\di(x, z))}{t} \geq-\di(x, y)(1-\rho).$$
 \end{itemize}

 Since $x\in E_r$ there exist $y\in B_r(y^*)$ such that $(x,y)\in \Gamma$. Denote $d:=\di(x,y)$, let $\gamma$ be the geodesic connecting $x$ and $y$, and let $\varepsilon=\left(\frac{(\rho-\delta+\rho\delta) d}{2 (3+2\delta)}\right)$. Up to taking a smaller $\varepsilon$, we can assume that $\varepsilon< d$. Consider the cone $C_{\gamma, \varepsilon}$. Since $\mu$ is densely scattered and $x$ is a $\mu$-density point of $E_r$, for any $D<r$ there exist some $\tilde{x} \in C_{\gamma, \varepsilon}\cap E_r$ such that $\di (x, \tilde{x})=d_0<D$ (that is $d_0$ can be made arbitrarily small), and also some $z\in B_r(z^*)$ such that $(\tilde{x}, z) \in \Gamma$. Due to the construction of $r$, we can guarantee that $x,\tilde{x}, y, z$ are all distinct. We refer the reader to Figure \ref{img: thm:NonBranch+ DenselyScattered -> !OTM} for a graphic representation.

 \begin{figure}[h]
   \centering
\begin{tikzpicture}
    \filldraw[fill=black!10!white](3,3.1) circle (1.2);
    \fill[fill=black!10!white] (-3,0)..controls (-2,0.25) and (-1,1.36).. (0.1,2.35)..controls (1,3.25) and (2,4.1).. (2.8,4.29)-- (3,3) --(3,1.9) ..controls (2,1.75) and (1,1.16) .. (0.2,0.88) ..controls (-1,0.35) and (-2,-0.25)..(-3,0);
    \draw (-3,0)..controls (-2,0.25) and (-1,1.36).. (0.1,2.35)..controls (1,3.25) and (2,4.1).. (2.8,4.29);
    \draw(3,1.9) ..controls (2,1.75) and (1,1.16) .. (0.2,0.88) ..controls (-1,0.35) and (-2,-0.25)..(-3,0);
    \draw[thick, blue] (-3,0)..controls (-1,0) and (1,3).. (3,3);
    \draw[thick, red] (-3,0)..controls (-2,-0.2) and (-1,0.5).. (0.1,1);
    \draw[thick] (2.5,-1.5)..controls (2,-1.4) and (1,1)..(0.1,1);
    \draw[thick,mygreen] (0.1,1)-- (0.3,1.7);
    \filldraw[red] (0.1,1) circle (2pt);
    \node at (-0.2,0.55)[label=east:$\color{red}\tilde x$] {};
    \node at (-1.2,0.25)[label=south:$\color{red}d_0$] {};
    \node at (2,3.2)[label=east:$\color{blue}\gamma$] {};
    \node at (-1,3.1)[label=east:$\color{black!30!white}C_{\gamma,\varepsilon}$] {};
    \node at (0.18,1.55)[label=north:$\color{blue}\gamma(t)$] {};
    \filldraw[blue] (0.3,1.7) circle (2pt);
    \filldraw[blue] (-3,0) circle (2pt);
    \node at (-3,0)[label=west:$\color{blue}x$] {};
    \node at (3,3)[label=east:$\color{blue}y$] {};
    \filldraw[blue] (3,3) circle (2pt);
    \node at (1.3,1.37)[label=west:$\color{mygreen}\leq \varepsilon t$] {};
    \filldraw (2.5,-1.5) circle (2pt);
    \node at (2.5,-1.5)[label=east:$z$] {};
\end{tikzpicture}
    \caption{Picture of the strategy to prove Theorem \ref{thm:NonBranch+ DenselyScattered -> !OTM}.}
    \label{img: thm:NonBranch+ DenselyScattered -> !OTM}
\end{figure}
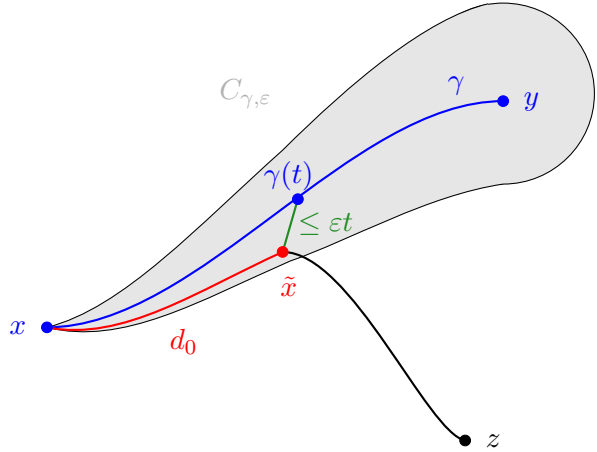

 Now we construct some $\tilde{x}$ so that the pairs $(x,y),(\tilde{x},z) \in \Gamma$ contradict $c$-cyclical monotonicity.

For each $\tilde{x} \in C_{\gamma, \varepsilon}\cap E_r$, by the definition of $C_{\gamma,\varepsilon}$, there exist a smallest $t>0$ such that $\di(\tilde{x}, \gamma(t)) \leq \varepsilon t$. The following estimate then holds:
 \begin{equation*}
     \begin{aligned}
\mathrm{Cost}_1 & :=\di^2(x, y)-\di^2(\tilde{x}, y) \\
& =(\di(x, y)+\di(\tilde{x}, y))(\di(x, y)-\di(\tilde{x}, y)) \\
& \geq(2 \di(x, y)-\di(x, \tilde{x}))(\di(x, y)-\di(\tilde{x}, y)) \\
& \geq\left(2 d-d_0\right)(t \di(x, y)-t \varepsilon) \\
& =2 t\left(d-\frac{d_0}{2}\right)(d-\varepsilon),
\end{aligned}
 \end{equation*}
  where we used triangular inequality and the observation: $\di (\tilde{x}, y) \leq \di (\tilde{x}, \gamma(t))+ \di(\gamma(t), y) \leq \varepsilon t+(1-t) d= \di(x, y)-t(d-\varepsilon) < \di(x, y).$ Notice that, up to taking $d_0$ sufficiently small, we have that $2d-d_0 > 0$. 

On the other hand, we note that 
$$\begin{aligned}
-\mathrm{Cost}_2 &:=\di ^2(\tilde{x}, z)-\di^2(x, z) \\
& =(\di(x, z)+\di(\tilde{x}, z))(\di(\tilde x,z)-\di(\gamma(t), z) + \di(\gamma(t),z) - \di(x, z)) \\
& \geq (\di(x, z)+\di(\tilde{x}, z))(-t \di(x, y)(1-\rho) - o(t) -\varepsilon t) \\
& =- t(\di(x, z)+\di(\tilde{x}, z))(d(1-\rho)+\varepsilon + o(1)) \\
& \geq-2 t\left((1+\delta) d+\frac{d_0}{2}\right)(d(1-\rho)+\varepsilon+ o(1)) \quad \text{as }t\to 0,
\end{aligned}$$
where we used both Fact III and Fact IV.
Observing that 
\begin{equation}
     d_0 = \di(x,\tilde{x}) \geq \di(x, \gamma(t)) - \di( \gamma(t), \tilde{x}) \geq (d-\varepsilon) t,
\end{equation}
we deduce that, up to taking $d_0$ sufficiently small, $t$ is sufficiently small and we have
$$-\mathrm{Cost}_2\geq-2 t\left((1+\delta) d+\frac{d_0}{2}\right)(d(1-\rho)+2\varepsilon) .$$

Summing the two estimates, we obtain 
$$\mathrm{Cost}_1 - \mathrm{Cost}_2 \geq 2 t\left(\left(d-\frac{d_0}{2}\right)(d-\varepsilon)-\left((1+\delta) d+\frac{d_0}{2}\right)(d(1-\rho)+2\varepsilon)\right):=2 t F\left(d_0\right).$$ 
For the selections of $\delta$ and $\varepsilon$, $F(0)=d(d-\varepsilon)-(1+\delta) d(d(1-\rho)+2\varepsilon)>0$. Also, $F$ changes continuously as $d_0$ decreases to $0$. Recall that $d_0= \di(x, \tilde{x})$ can be made arbitrarily small, with $t>0$, this means $\mathrm{Cost}_1-\mathrm{Cost}_2 > 0$ for some particular $d_0$. That is, for some selection of $\tilde{x}$ (and $z$), $\di^2(x, y)+\di^2(\tilde{x}, z)>\di^2(\tilde{x}, y)+\di^2(x, z)$, contradicting $c$-cyclical monotonicity of $\Gamma$.

\textbf{Case 2}:  $\di(x^*, y^* ) > \di(x^*, z^*)$

Because $\di$ is continuous there exist some small $1> \delta > 0$ and $r_1>0$ so that $(1-\delta) \di(x, y) \geq \di(x, z)$ for all $x\in B_{r_1}(x^*)$ $y \in B_{r_1}\left(y^*\right), z \in B_{r_1}\left(z^*\right)$. Take $r_2:=\frac{1}{3} \min \left\{d\left(x^*, y^*\right), d\left(y^*, z^*\right)\right\}$ and set $r=r_1 \wedge r_2.$

By the same argument of Case 1, there is some $x\in E_r$ being a density point of $E_r$ and such that $(1-\delta) \di(x, y) \geq \di(x, z)$ for all $y \in B_r\left(y^*\right), z \in B_r\left(z^*\right)$. 
Since $\pi_x\left(B_r\left(y^*\right)\cap \Gamma_x\right)>0$, there exist $y\in B_r\left(y^*\right)$ such that $(x,y)\in \Gamma$. Denote $d:=\di(x,y)$, let $\gamma$ be the geodesic connecting $x$ and $y$, and take $\varepsilon = \left(\frac{\delta d}{4-2\delta}\right) $. Following the same arguments as in Case 1, for any $D<r$ there exist some $\tilde{x} \in C_{\gamma, \varepsilon}\cap E_r$ such that $\di(x, \tilde{x})=d_0<D$, and some $z\in B_r(z^*)$ such that $(\tilde{x}, z) \in \Gamma$, where $x,\tilde{x}, y, z$ are all distinct. Now we show that, for suitable $\tilde{x}$ and $z$ (constructed as just explained), the pairs $(x,y),(\tilde{x},z) \in \Gamma$ contradict $c$-cyclical monotonicity.

Given $\tilde{x} \in C_{\gamma, \varepsilon}\cap E_r$, let $t$ be the smallest time so that $\di(\tilde{x}, \gamma(t)) \leq \varepsilon t$. Then the following estimate holds:

$$\mathrm{Cost}_1 :=\di^2(x, y)-\di^2(\tilde{x}, y) \geq 2 t\left(d-\frac{d_0}{2}\right)(d-\varepsilon),$$
as in Case 1. On the other hand, we obtain that 
$$\begin{aligned}
 -\mathrm{Cost}_2 &:=\di^2(\tilde{x}, z)-\di^2(x, z) \\
& =(\di(x, z)+\di(\tilde{x}, z))(\di(\tilde x, z)-\di(x, z)) \\
& \geq - (\di(x, z)+\di(\tilde{x}, z)) d_0 \\
& \geq- t(\di(x, z)+\di(\tilde{x}, z)) (d+\varepsilon) \\
& \geq-2 t\left(\di(x, z)+\frac{d_0}{2}\right)(d+\varepsilon) \\
& \geq-2 t\left((1-\delta) d+\frac{d_0}{2}\right)(d+\varepsilon),
\end{aligned}$$
where we used that $d_0 \leq td + \varepsilon t$.

Summing the two estimates, we conclude $$\mathrm{Cost}_1 - \mathrm{Cost}_2 \geq 2 t\left\{\left(d-\frac{d_0}{2}\right)(d-\varepsilon)-\left((1-\delta) d+\frac{d_0}{2}\right)(d+\varepsilon)\right\}=2 t G\left(d_0\right).$$ 
For the selections of $\delta$ and $\varepsilon$, $G(0)>0$. Also, $G$ changes continuously as $d_0$ decreases to $0$. Recall $d_0$ can be made arbitrarily small, with $t>0$, this means $\mathrm{Cost}_1 - \mathrm{Cost}_2 > 0$ for some particular $d_0$, that is, for some selection of $\tilde{x}$ (and $z$), $\di^2(x, y)+\di^2(\tilde{x}, z)>\di^2(\tilde{x}, y)+\di^2(x, z)$ contradicting $c$-cyclical monotonicity of $\Gamma$.
\end{proof}

Combining Theorem \ref{thm:NonBranch+ DenselyScattered -> !OTM} with Proposition \ref{prop:doublingtoDS}, we immediately get the following corollary.

\begin{corollary}\label{cor:final}
    Let $(\X,\di)$ be a locally-uniformly non-branching metric space, equipped with a locally doubling reference measure $\m$. 
  Let $c(x,y)=\di^2(x,y)$. Then, for every $\mu,\nu\in \Prob(\X)$ such that $\mu\ll\m$, if $\OptPlans(\mu,\nu)\neq \emptyset$, there exists a unique optimal transport plan $\pi \in \OptPlans(\mu,\nu)$ and it is induced by a map. 
\end{corollary}

\bibliographystyle{alpha} 
\bibliography{bibliography}

\end{document}